\documentclass[12pt]{amsart}
\usepackage{amsmath,amssymb,amsbsy,amsfonts,amsthm,latexsym,
                        amsopn,amstext,amsxtra,euscript,amscd,mathrsfs,color,bm}
                        
\usepackage{float} 
\usepackage[english]{babel}
\usepackage{url}
\usepackage[colorlinks,linkcolor=blue,anchorcolor=blue,citecolor=blue,backref=page]{hyperref}

\usepackage[margin=3.5cm]{geometry}
\usepackage{color}
\usepackage[norefs,nocites]{refcheck}

\renewcommand*{\backref}[1]{}
\renewcommand*{\backrefalt}[4]{%
    \ifcase #1 (Not cited.)%
    \or        (p.\,#2)%
    \else      (pp.\,#2)%
    \fi}

\begin{document}

\newtheorem{theorem}{Theorem}
\newtheorem{lemma}[theorem]{Lemma}
\newtheorem{example}[theorem]{Example}
\newtheorem{algol}{Algorithm}
\newtheorem{corollary}[theorem]{Corollary}
\newtheorem{prop}[theorem]{Proposition}
\newtheorem{definition}[theorem]{Definition}
\newtheorem{question}[theorem]{Question}
\newtheorem{problem}[theorem]{Problem}
\newtheorem{remark}[theorem]{Remark}
\newtheorem{conjecture}[theorem]{Conjecture}

\newcommand{\commM}[1]{\marginpar{%
\begin{color}{red}
\vskip-\baselineskip 
\raggedright\footnotesize
\itshape\hrule \smallskip M: #1\par\smallskip\hrule\end{color}}}

\newcommand{\commA}[1]{\marginpar{%
\begin{color}{blue}
\vskip-\baselineskip 
\raggedright\footnotesize
\itshape\hrule \smallskip A: #1\par\smallskip\hrule\end{color}}}
\def\xxx{\vskip5pt\hrule\vskip5pt}


\def\cA{{\mathcal A}}
\def\cB{{\mathcal B}}
\def\cC{{\mathcal C}}
\def\cD{{\mathcal D}}
\def\cE{{\mathcal E}}
\def\cF{{\mathcal F}}
\def\cG{{\mathcal G}}
\def\cH{{\mathcal H}}
\def\cI{{\mathcal I}}
\def\cJ{{\mathcal J}}
\def\cK{{\mathcal K}}
\def\cL{{\mathcal L}}
\def\cM{{\mathcal M}}
\def\cN{{\mathcal N}}
\def\cO{{\mathcal O}}
\def\cP{{\mathcal P}}
\def\cQ{{\mathcal Q}}
\def\cR{{\mathcal R}}
\def\cS{{\mathcal S}}
\def\cT{{\mathcal T}}
\def\cU{{\mathcal U}}
\def\cV{{\mathcal V}}
\def\cW{{\mathcal W}}
\def\cX{{\mathcal X}}
\def\cY{{\mathcal Y}}
\def\cZ{{\mathcal Z}}

\def\C{\mathbb{C}}
\def\F{\mathbb{F}}
\def\K{\mathbb{K}}
\def\G{\mathbb{G}}
\def\Z{\mathbb{Z}}
\def\R{\mathbb{R}}
\def\Q{\mathbb{Q}}
\def\N{\mathbb{N}}
\def\M{\textsf{M}}
\def\U{\mathbb{U}}
\def\P{\mathbb{P}}
\def\A{\mathbb{A}}
\def\p{\mathfrak{p}}
\def\n{\mathfrak{n}}
\def\X{\mathcal{X}}
\def\x{\textrm{\bf x}}
\def\w{\textrm{\bf w}}
\def\ovQ{\overline{\Q}}
\def\rank#1{\mathrm{rank}#1}
\def\wf{\widetilde{f}}
\def\wg{\widetilde{g}}
\def\comp{\hskip -2.5pt \circ  \hskip -2.5pt}
\def\({\left(}
\def\){\right)}
\def\[{\left[}
\def\]{\right]}
\def\<{\langle}
\def\>{\rangle}

\def\gen#1{{\left\langle#1\right\rangle}}
\def\genp#1{{\left\langle#1\right\rangle}_p}
\def\genPs{{\left\langle P_1, \ldots, P_s\right\rangle}}
\def\genPsp{{\left\langle P_1, \ldots, P_s\right\rangle}_p}

\def\e{e}

\def\eq{\e_q}
\def\fh{{\mathfrak h}}

\def\lcm{{\mathrm{lcm}}\,}

\def\l({\left(}
\def\r){\right)}
\def\fl#1{\left\lfloor#1\right\rfloor}
\def\rf#1{\left\lceil#1\right\rceil}
\def\mand{\qquad\mbox{and}\qquad}

\def\jt{\tilde\jmath}
\def\ellmax{\ell_{\rm max}}
\def\llog{\log\log}

\def\m{{\rm m}}
\def\ch{\hat{h}}
\def\GL{{\rm GL}}
\def\Orb{\mathrm{Orb}}
\def\Per{\mathrm{Per}}
\def\Preper{\mathrm{Preper}}
\def \S{\mathcal{S}}
\def\vec#1{\mathbf{#1}}
\def\ov#1{{\overline{#1}}}
\def\Gal{{\rm Gal}}

\newcommand{\bfalpha}{{\boldsymbol{\alpha}}}
\newcommand{\bfomega}{{\boldsymbol{\omega}}}

\newcommand{\Ch}{{\operatorname{Ch}}}
\newcommand{\Elim}{{\operatorname{Elim}}}
\newcommand{\proj}{{\operatorname{proj}}}
\newcommand{\h}{{\operatorname{h}}}

\newcommand{\hh}{\mathrm{h}}
\newcommand{\aff}{\mathrm{aff}}
\newcommand{\Spec}{{\operatorname{Spec}}}
\newcommand{\Res}{{\operatorname{Res}}}

\numberwithin{equation}{section}
\numberwithin{theorem}{section}

\def\house#1{{%
    \setbox0=\hbox{$#1$}
    \vrule height \dimexpr\ht0+1.4pt width .5pt depth \dimexpr\dp0+.8pt\relax
    \vrule height \dimexpr\ht0+1.4pt width \dimexpr\wd0+2pt depth \dimexpr-\ht0-1pt\relax
    \llap{$#1$\kern1pt}
    \vrule height \dimexpr\ht0+1.4pt width .5pt depth \dimexpr\dp0+.8pt\relax}}

\newcommand{\Address}{{
\bigskip
\footnotesize
\textsc{Centro di Ricerca Matematica Ennio De Giorgi, Scuola Normale Superiore, Pisa, 56126, Italy}\par\nopagebreak
\textit{E-mail address:} \texttt{marley.young@sns.it}
}}

\title[]
{On multiplicative dependence between elements of polynomial orbits}

\author[Marley Young] {Marley Young}

\subjclass[2020]{37F10, 37P15, 11N25, 11D41}

\begin{abstract}
We classify the pairs of polynomials $f,g \in \C[X]$ having orbits satisfying infinitely many multiplicative dependence relations, extending a result of Ghioca, Tucker and Zieve. Moreover, we show that given $f_1,\ldots, f_n$ from a certain class of polynomials with integer coefficients, the vectors of indices $(m_1,\ldots,m_n)$ such that $f_1^{m_1}(0),\ldots,f_n^{m_n}(0)$ are multiplictively dependent are sparse. We also classify the pairs $f,g \in \Q[X]$ such that there are infinitely many $(x,y) \in \Z^2$ satisfying $f(x)^k=g(y)^\ell$ for some 
(possibly varying) non-zero integers $k,\ell$.
\end{abstract}

\maketitle

\section{Introduction}

In complex dynamics, a topic of great importance is the behaviour of complex numbers $x$ as they are iterated under a polynomial $f \in \C[X]$. That is, we are interested in the \emph{orbits} $\cO_f(x) := \{ x,f(x),f(f(x)),\ldots \}$, and how they interact with each other, and the structure of the polynomial $f$. For example, the orbits of critical points largely determine the features of the global dynamics of the map, see \cite[\textsection 9]{Blanchard}. Points with finite orbit, called \emph{preperiodic points}, also provide a lot of information \cite[\textsection 3]{Blanchard}. On the other hand, less is known about the interaction between orbits of distinct polynomials. However, in \cite{GTZ2}, Ghioca, Tucker and Zieve prove the following remarkable fact.

\begin{theorem}\cite[Theorem~1.1]{GTZ2} \label{thm:GTZ}
Let $f,g \in \C[X]$ be polynomials which are not linear. If there exist $x,y \in \C$ such that the intersection $\cO_f(x) \cap \cO_g(y)$ is infinite, then $f$ and $g$ share a common iterate.
\end{theorem}

One can interpret Theorem~\ref{thm:GTZ} as describing when $\cO_f(x) \times \cO_g(y)$ has infinite intersection with the diagonal $\Delta = \{(z,z): z \in \C \}$. The conclusion says that this occurs precisely when there exist positive integers $n,m$ such that the line $\Delta$ is preserved by the map $(f^n,g^m):\C^2 \to \C^2$, where $f^n$ denotes the $n$-fold composition
$$
f^n = \underbrace{f \circ \cdots \circ f}_{n \text{ times}}.
$$
In fact, this generalises to several polynomials and arbitrary lines, see \cite[Theorem~1.3]{GTZ2}. This resolves a special case of a version of the so-called \emph{dynamical Mordell-Lang conjecture}. Namely, Ghioca and Tucker conjectured the following \cite[Conjecture~1.3]{GTZ1}.

\begin{conjecture}
Let $f_1,\ldots,f_k$ be polynomials in $\C[X]$, and let $V$ be a subvariety of the affine space $\A^k$ which contains no positive dimensional subvariety of $\A^k$ that is periodic under the action of $(f_1,\ldots,f_k)$ on $\A^k$. Then $V(\C)$ has finite intersection with each orbit of $(f_1,\ldots,f_k)$ on $\A^k$.
\end{conjecture}

This conjecture fits into Zhang's far-reaching system of dynamical conjectures \cite{Z}, and has more general formulations, see \cite[Question~1.6]{GTZ2}.

In this paper, we prove a generalisation of Theorem~\ref{thm:GTZ}, where we replace the diagonal $\Delta$ with a set $\mathscr{M}_{2,1}(\C)$ of points in $\C^2$ whose coordinates satisfy certain multiplicative relations. Instead of $f$ and $g$ necessarily having a common iterate, we will see that in this situation, some $\hat f$ and $\hat g$, belonging to explicitly computable families of polynomials forming certain semiconjugacies with $f$ and $g$ respectively, will have a common iterate. 

Let $n$ be a positive integer and let $\bm{\nu}=(\nu_1,\ldots,\nu_n) \in (\C^\times)^n$. We say that $\bm{\nu}$ is \emph{multiplicatively dependent} if all its entries are non-zero and there is a non-zero vector $\bm{k} = (k_1,\ldots,k_n) \in \Z^n$ for which
\begin{equation*}
\bm{\nu}^{\bm{k}} = \nu_1^{k_1} \cdots \nu_n^{k_n} = 1.
\end{equation*}
Otherwise we say that $\bm{\nu}$ is \emph{multiplicatively independent}. By convention we say that a vector $\bm{\nu} \in \C^n$ which has at least one zero entry is neither multiplicatively dependent nor independent. For a subset $T$ of $\C$, we denote by $\mathscr{M}_n(T)$ the set of multiplicatively dependent vectors with coordinates in $T$.

For $\bm{\nu} \in (\C^\times)^n$, we define $r$, the \emph{multiplicative rank} of $\bm{\nu}$, in the following way. If $\bm{\nu}$ has a coordinate which is a root of unity, we put $r=0$; otherwise let $r$ be the largest integer with $1 \leq r \leq n$ for which any $r$ coordinates of $\bm{\nu}$ form a multiplicatively independent vector. Note that $0 \leq r \leq n-1$ whenever $\bm{\nu}$ is multiplicatively dependent. For a subset $T$ of $\C$, we denote by $\mathscr{M}_{n,r}(T)$ the set of multiplicatively dependent vectors of rank $r$ with coordinates in $T$, and note that
$$
\mathscr{M}_n(T) = \mathscr{M}_{n,0}(T) + \cdots + \mathscr{M}_{n,n-1}(T).
$$


In \cite{KSSS}, Konyagin, Sha, Shparlinski and Stewart studied the distribution of $\mathscr{M}_n(T)$ when $T \subseteq \C$ is of number theoretic interest. Note that $\mathscr{M}_n(\C)$ has zero Lebesgue measure, since it is a countable union of hypersurfaces. However, the results of \cite{KSSS} imply that $\mathscr{M}_n(\C)$ is dense in $\C^n$. In particular, $\mathscr{M}_{2,1}(\C)$ is dense in $\C^2$ (since $\mathscr{M}_{2,0}(\C)$ is clearly nowhere dense). Thus our generalisation of Theorem~\ref{thm:GTZ} (Theorem~\ref{thm:MDGTZ} below) shows that some product of orbits $\cO_f(x) \times \cO_g(y)$ merely having infinite intersection with a dense subset of $\C^2$ is enough to have strong implications on the structure of the pair $(f,g)$.

Moreover, the mutliplicative dependence of complex numbers, and in particular of algebraic numbers, has also been studied from various other aspects, and is a subject of independent interest in the contexts of algebraic geometry and arithmetic dynamics. In \cite{BMZ}, Bombieri, Masser and Zannier initiated the study of intersections of algebraic curves with proper algebraic subgroups of the multiplicative group $\G_{\mathrm{m}}^n$. Since such subgroups of $\G_{\mathrm{m}}^n$ are defined by finite sets of equations of the form $X^{k_1} \cdots X^{k_n}=1$ (see \cite[Corollary~3.2.15]{BG}), the paper \cite{BMZ}, which leads into the paradigm of ``unlikely intersections'', really concerns multiplicative dependence of points on a curve. Further to this, in \cite{OSSZ}, the authors obtain finiteness results for multiplicatively dependent values of rational functions in the maximal abelian extension of a number field $K$, and also show that under certain conditions on a rational function $f \in K(X)$, there are only finitely many $\alpha \in K$ such that $(f^n(\alpha),f^m(\alpha))$ is multiplicatively dependent for some distinct $m,n \geq 0$ (these results have since been extended to hold modulo (approximate) finitely generated groups, see \cite{BBGMOS2,BOSS}). That is, multiplicative relations within orbits of a polynomial or rational function have implications for its structure. Note also that for non-exceptional (semigroups of) polynomials over the cyclotomic closure of a number field, only finitely many orbits contain multiplicative relations of rank 0 (i.e. roots of unity), see \cite{C,O,OY}, following work of Dvornicich and Zannier \cite{DZ}.

In a similar vein, we show that multiplicative relations between orbits of distinct polynomials imply some relationship between the polynomials themselves. For a polynomial $f \in \C[X]$ and a point $\alpha \in \C$, denote by $f_\alpha$ the polynomial $f_\alpha(X)=\alpha f(\alpha^{-1}X)$.

\begin{theorem} \label{thm:MDGTZ}
Let $f,g \in \C[X]$ be polynomials which are neither linear nor monomials (i.e. of the form $aX^d$). Then there exist $x,y \in \C$ such that there are infinitely many pairs $(z,w) \in \cO_f(x) \times \cO_g(y)$ which are multiplicatively dependent of rank 1 if and only if there exist:
\begin{itemize}
\item integers $s,t \geq 0$,
\item coprime integers $k,\ell \geq 1$,
\item positive integers $i,j \leq 2$, where we can take $i=1$ (resp. $j=1$) unless $\ell=2$ (resp. $k=2$),
\item $\tilde f, \tilde g \in \C[X]$, and
\item a root of unity $\xi$,
\end{itemize}
such that $f^i(X)= X^s \tilde f(X)^\ell$, $g^j(X)=X^t \tilde g(X)^k$, and $\hat f$ and $\hat g_\xi$ share a common iterate, where $\hat f(X)=X^s \tilde f(X^\ell)$, $\hat g(X)=X^t \tilde g(X^k)$.
\end{theorem}

We ignore multiplicatively dependent pairs of rank 0 since infinitely many can always be found. Indeed, let $x \in \C$ be such that $f(x)$ is a root of unity, and suppose $y$ is not preperiodic for $g$. Then $\{(f(x),g^n(y)) : n \in \Z^+\}$ is an infinite set of multiplicatively dependent points. We can similarly omit the case where $f$ or $g$ is a monomial. Suppose $f(X)=aX^k$, $a \in \C^\times$. If $a$ is a root of unity, let $x \in \C$ be not a root of unity. Then $(f^n(x),x)=(a^{(k^n-1)/(k-1)} x^{k^n},x)$ is multiplicatively dependent of rank 1 for all $n \in \Z^+$. If $a$ is not a root of unity, let $y \in \C$ be such that $g(y)=a$. Then $(f^n(a),g(y))=(a^{k^n+(k^n-1)/(k-1)},a)$ is multiplicatively dependent of rank 1 for all $n \in \Z^+$.

The proof of Theorem~\ref{thm:GTZ} makes crucial use of the work of Bilu and Tichy \cite{BT} on Diophantine equations in separated variables. Bilu and Tichy completely classify the polynomial pairs $f,g \in \Q[X]$ such that $f(x)=g(y)$ has infinitely many integral solutions (or more generally rational solutions with bounded denominator). As an aside, with an additional argument using bounds on solutions to superelliptic equations given in \cite{BBGMOS}, we are able to characterise the pairs of polynomials $f,g \in 
\Z[X]$ for which $(f(x),g(y))$ is multiplicatively dependent of rank 1 for infinitely many $x,y \in \Z$.

\begin{theorem} \label{thm:MDBT1}
Let $f,g \in \Z[X]$ each have at least two distinct roots. Then there are only finitely many $(x,y) \in \Z^2$ such that $(f(x),g(y))$ is multiplicatively dependent of rank 1 unless for some $(k,\ell)$ in an explicitly computable (in terms of $f$ and $g$) finite subset of $\Z^+ \times \Z^+$, we have $f(X)^k = \varphi(f_1(\lambda(X)))$ and $\pm g(X)^\ell = \varphi(g_1(\mu(X)))$, where $\lambda, \mu \in \Q[X]$ are linear, $\varphi \in \Q[X]$, and $(f_1,g_1)$ is a standard pair (see \textsection \ref{sec:BT}), such that $f_1(x)=g_1(y)$ has infinitely many rational solutions with bounded denominator.
\end{theorem} 

We give a version of this result over a ring of $S$-integers in an arbitrary number field in \textsection \ref{sec:BT}. A similar approach is also key to our proof of Theorem~\ref{thm:MDGTZ}, after we use a specialisation argument to reduce to the case where $f,g,x,y$ are defined over a number field.

An interesting direction for future research would be an extension of Theorem~\ref{thm:MDGTZ} to more than two polynomials.

\begin{question} \label{Q}
Let $f_1,\ldots,f_n \in \C[X]$ be polynomials which are not linear or monomials. If there exist $x_1,\ldots,x_n \in \C$ such that there are infinitely many $n$-tuples $(u_1,\ldots,u_n) \in \cO_{f_1}(x_1) \times \cdots \times \cO_{f_n}(x_n)$ which are multiplicatively dependent of rank at least 1, is there some multiplicative relation between some (possibly semiconjugated) iterates of the $f_i$?
\end{question}

This is likely a very difficult question, even in the case $n=3$. Indeed, we recall the following conjecture of Schinzel and Tijdeman \cite{SchTij}, which remains open (though it is known to hold assuming the $abc$ conjecture \cite{Walsh}).

\begin{conjecture}
If  a polynomial $f$ with rational coefficients has at least three simple zeros then the equation $y^3 z^2=f(x)$ has only finitely many solutions in integers $x,y,z$ with $yz \neq 0$.
\end{conjecture}

The depth of this conjecture is evident as it implies the existence of infinitely many primes $p$ for which $2^{p-1} \not\equiv 1 \pmod{p^2}$. We would require an even stronger generalisation just to mimic the part of our proof of Theorem~\ref{thm:MDGTZ} which uses bounds on solutions to superelliptic equations. Furthermore, another key part of the proof of Theorem~\ref{thm:GTZ} is a collection of intricate polynomial decomposition results. The polynomial decomposition work required to tackle Question~\ref{Q} would be even more complicated, unless a totally different approach were discovered.

Hence, for now we consider the following statistical variant. Let $K$ be a field and suppose we have $F = (f_1,\ldots,f_n) \in K(X)^n$, $\bm{x} = (x_1,\ldots, x_n) \in K^n$, and $N \in \Z^+$, and let $M_{F,\bm{x}}(N)$ denote the number of $n$-tuples of integers $(m_1,\ldots,m_n) \in [1,N]^n$ such that $(f_1^{m_1}(x_1), \ldots, f_n^{m_n}(x_n))$ is multiplicatively dependent. Given integer polynomials which generate certain kinds of divisibility sequences, we are able to show that the tuples of indices which give a multiplicative dependence relation are at least sparse.

\begin{theorem} \label{thm:SparseIter}
Let $F = (f_1,\ldots,f_n) \in \Z[X]^n$, and $\bm{x} = (x_1,\ldots,x_n) \in \Z^n$ such that for each $1 \leq i \leq n$, $f_i$ is not linear nor a monomial, and $\{ f_i^m(x_i) \}_{m \geq 1}$ is an unbounded rigid divisibility sequence. Then $M_{F, \bm{x}}(N) \ll N^n/\log N$, with the implied constant depending only on $F$ and $\bm{x}$.
\end{theorem}

Above, and throughout the paper, the assertion $U \ll V$ is equivalent to $|U| \leq cV$ with some positive constant $c$. We recall the notion of a rigid divisibility sequence in \textsection \textsection \ref{subsec:RDS}, and give examples of classes of polynomials which generate them.  

Furthermore, let us note that in the special case $f_1 = \cdots = f_n$, $x_1=\cdots=x_n$, we can treat rational functions over a number field or function field (assuming the $abc$-conjecture of Masser-Oesterl\'{e}-Szpiro \cite{V}, or some conditions on $(f,x)$ in the number field case) and achieve a better bound by using results on primitive prime divisors in dynamical sequences.

\begin{theorem} \label{thm:MDZsig}
Let $K$ be a number field or characteristic zero function field of transcendence degree 1, let $f \in K(X)$ be a rational function which is not linear or a monomial, and let $x \in K$ be a point which is not preperiodic for $f$. Let $F=(f,\ldots,f) \in K(X)^n$, and let $\bm{x}=(x,\ldots,x) \in K^n$. If $K$ is a number field, assume that either $K$ satisfies the $abc$-conjecture, or that $0$ is periodic for $f$ and $f$ does not vanish to order $\deg f$ at 0. Then $M_{F,\bm{x}}(N) = \Theta(N^{n-1})$.
\end{theorem}

Here $U=\Theta(V)$ means $U \ll V$ and $V \ll U$. It is clear that $M_{F,\bm{x}}(F) \gg N^{n-1}$ in this case, since we can just take $m_i = m_j$ for some $i \neq j$ and let all the other coordinates run freely over $[1,N]$.

This paper is structured as follows: In \textsection \ref{sec:prelim} we present some useful tools and preliminary results, including the polynomial $ABC$-theorem, and bounds on integral solutions to superelliptic equations. In \textsection \ref{sec:BT}, we state Bilu and Tichy's results on Diophantine equations in separated variables and prove a more general form of Theorem~\ref{thm:MDBT1}. The proof of Theorem~\ref{thm:MDGTZ} is given in \textsection \ref{sec:GTZ}. Finally, in \textsection \ref{sec:Sparse} we give some background on rigid divisibility sequences and primitive prime divisors in arithmetic dynamics, and prove Theorems~\ref{thm:SparseIter}~and~\ref{thm:MDZsig}.

\subsection*{Acknowledgements} The author is very grateful to Alina Ostafe and Igor Shparlinski for comments on initial drafts of the paper. The latter also gave the suggestion of considering Theorem~\ref{thm:MDZsig}.

\section{Preliminaries} \label{sec:prelim}

We collect here some notions and results which will be useful subsequently.

\subsection{The polynomial ABC-theorem}

We recall the polynomial $ABC$-theorem (proved first by Stothers \cite{Stothers}, then independently by Mason \cite{Mason} and Silverman \cite{Silverman}).

\begin{lemma} \label{lem:ABC}
Let $K$ be a field and let $A,B,C \in K[X]$ be relatively prime polynomials such that $A+B+C=0$ and not all of $A, B$ and $C$ have vanishing derivative. Then
$$
\deg \mathrm{rad}(ABC) \geq \max \{ \deg A, \deg B, \deg C \} + 1,
$$
where, for $f \in K[X]$, $\mathrm{rad} \, f$ is the product of the distinct monic irreducible factors of $f$.
\end{lemma}

\subsection{Bounds on solutions to superelliptic equations} \label{subsec:BBGMOS}

Let $K$ be a number field, and let $S$ be a finite set of places of $K$ containing the archimedean ones. Denote by $\cO_S$ the ring of $S$-integers in $K$. Let
$$
f(X) = a_0 X^n + a_1 X^{n-1} + \cdots + a_n \in \mathfrak{o}_S[X]
$$
be a polynomial of degree $n \geq 2$. Let $b$ be a non-zero element of $\mathfrak{o}_S$, $m \geq 2$ and integer and consider the equation
\begin{equation} \label{eq:Superell}
f(x)=by^m, \qquad x,y \in \mathfrak{o}_S.
\end{equation}
Assume that in some finite extension of $K$, $f \in \mathfrak{o}_S[X]$ factorises as
$$
f(X)=a_0 \prod_{i=1}^r (X-\alpha_i)^{e_i}
$$
with distinct $\alpha_1,\ldots,\alpha_r$. Put
$$
m_i = \frac{m}{\gcd(m,e_i)}, i=1,\ldots,r,
$$
and assume that $m_1 \geq m_2 \geq \cdots \geq m_r$. Improving a classical result of Siegel \cite{Siegel}, LeVeque \cite{LeVeque} showed that if $S$ consists only of infinite places and $b=1$ then \eqref{eq:Superell} has only finitely many solutions, provided $(f,m)$ satisfies the so-called \emph{LeVeque condition}, where the tuple $(m_1,\ldots,m_r)$ is of the form
\begin{equation} \label{eq:LeVeque}
m_1 \geq 3, m_2 \geq 2 \quad \text{or} \quad m_1=m_2=m_3=2.
\end{equation}

Moreover, in the case $K=\Q$, $\mathfrak{o}_S=\Z$, $b=1$, Schinzel and Tijdeman \cite{SchTij} treated \eqref{eq:Superell} with the exponent $m$ unknown, showing that there are no solutions if $m$ exceeds an effectively computable bound depending only on $f$. All of these results are based on Baker's theory of linear forms in logarithms of algebraic numbers, and have since been generalised over number fields and (in the case of $f$ having no multiple roots) over finitely generated domains over $\Z$ \cite{BEG}. We will state below a version from \cite{BBGMOS} which is useful for our purposes.

For $\alpha \in K$, define the \emph{$S$-norm} of $\alpha$ by
$$
N_S(\alpha) = \prod_{v \in S} |\alpha|_v,
$$
where $| \cdot |_v$ is the normalised absolute valuation of $K$ at the place $v$. When $\alpha \neq 0$, we denote by $h(\alpha)$ the absolute logarithmic Weil height of $\alpha$, see \cite{BG}.

The following is a consequence of \cite[Theorems~2.1~and~2.2]{BBGMOS}.

\begin{theorem} \label{thm:Superell}
Suppose \eqref{eq:Superell} has a solution $(x,y)$ with $y \notin \mathfrak{o}_S^* \cup \{0\}$. Then
$$
m \ll \log^*(N_S(b)) \log^* \log^* N_S(b),
$$
Moreover, if $(f,m)$ satisfies the LeVeque condition \eqref{eq:LeVeque}, then all solutions $(x,y)$ of \eqref{eq:Superell} satisfy
$$
h(x) \ll \begin{cases} N_S(b)^{16r^3} & \text{if } m=2, \\ N_S(b)^{4m^8r^2} & \text{otherwise,} \end{cases}
$$
where $\log^* x = \max(1,\log x)$ and the implied constants depend only on $K$, $S$ and $f$.
\end{theorem}

In order to later apply Theorem~\ref{thm:Superell} to polynomial iterates in the proof of Theorem~\ref{thm:MDGTZ}, we show that given $m \geq 2$, unless a polynomial $f$ is of a special shape, some iterate of $f$ satisfies the LeVeque condition with respect to $m$.

\begin{lemma} \label{lem:IterLeVeque}
Let $f \in \mathfrak{o}_S[X]$ be a polynomial which is not linear nor a monomial, and let $m \geq 2$ be an integer. If $f$ is not of the form $f(X) = X^s p(X)^m$ for some integer $s \geq 0$, and $p \in \overline K[X]$, then either $m=2$ and $f^2$ is of this form, or $(f^j,m)$ satisfies the LeVeque condition for some $j \leq 6$.
\end{lemma}

\begin{example}
There do exist polynomials $f$ such that $f^2$ is of the exceptional form. For example, let $f(X) = X^3-6iX^2-9X+4i=(X-4i)(X-i)^2$. Then
$$
f^2(X) = X(X^4-9iX^3-27X^2+30iX+9)^2.
$$
\end{example}

\begin{proof}[Proof of Lemma~\ref{lem:IterLeVeque}]
Suppose $f$ is not of the given form, and that $(f,m)$ does not satisfy the LeVeque condition. Then either 
\begin{enumerate}
\item[(i)] $f(X) = (X-\beta)^s p(X)^m$ for some $s \in \Z^+$ not divisible by $m$, $\beta \neq 0 \in \overline K$, and $p \in \overline K[X]$ with $p(\beta) \neq 0$; or
\item[(ii)] $f(X)=(X-\alpha)^r(X-\beta)^s p(X)^m$ for some $\alpha \neq \beta \in \overline K$, $p \in \overline K[X]$ with $p(\alpha),p(\beta) \neq 0$, and $r,s \in \Z^+$ with $\gcd(r,m)=\gcd(s,m)=m/2$. 
\end{enumerate} 
In the first case, we have
\begin{equation} \label{eq:deg2}
\deg f = s + m \deg p,
\end{equation}
and
\begin{equation} \label{eq:degrad2}
\deg \mathrm{rad} \, f = 1+\deg \mathrm{rad} \, p \leq 1 + \deg p.
\end{equation}
By the Lemma~\ref{lem:ABC} with $A=-f(X)$, $B=f(X)-\beta$ and $C=\beta$, we have
$$
\deg \mathrm{rad}(f(X)(f(X)-\beta)) \geq \deg f + 1,
$$
and so by \eqref{eq:deg2} and \eqref{eq:degrad2},
\begin{align} \label{eq:appABC2}
\deg \mathrm{rad}(f(X)-\beta) & \geq s+m \deg p + 1 - (1+ \deg p) \notag \\
& = s + (m-1) \deg p.
\end{align}
Now,
$$
f^2(X) = (f(X)-\beta)^s p(f(X))^m,
$$
so if $f(X)-\beta$ has at least two simple roots, we have reduced to the case (ii) (with $f^2$ in place of $f$). If $f(X)-\beta$ has $u \leq 1$ simple roots, then by \eqref{eq:deg2},
$$
\deg \mathrm{rad}(f(X)-\beta) \leq \frac{\deg f + u}{2} \leq \frac{u+s+m \deg p}{2}.
$$
This gives a contradiction with \eqref{eq:appABC2} if either $u=0$, $s > 1$ or $m > 2$ (in the latter case we also need $\deg p > 0$, but if $\deg p = 0$ then necessarily $s > 1$). If $m=2$ and $s=1$, we have $f(X) = (X-\beta) p(X)^2$. Suppose $f(X)-\beta$ has exactly one simple root, say $f(X)-\beta = (X-\gamma) q_1(X)^2$, and similarly that $f(X)-\gamma = (X-\delta) q_2(X)^2$. Note that $p$ and $q_1$ are coprime unless $\beta = \gamma$, in which case
$$
(X-\beta)q_1(X)^2 = f(X)-\beta = (X-\beta)p(X)^2 - \beta,
$$
and so
$$
(X-\beta)(p(X)^2-q_1(X)^2)=\beta,
$$
which is only possible if $p=q_1$ and $\beta=0$, a contradiction. Similarly $p$ and $q_2$ are coprime, and moreover $q_1$ and $q_2$ are coprime unless $\gamma = 0$, in which case $f^2(X) = Xq(X)^2$ for some polynomial $q$. Therefore we see that $p q_1 q_2 \mid f'$, but $\deg (p q_1 q_2) = 3(\deg f - 1)/2 > \deg f'$, a contradiction. Therefore either $f^2$ or $f^3$ must satisfy the LeVeque condition with 2 or be of the form (ii).

In the case (ii), we have
\begin{equation} \label{eq:deg1}
\deg f = r + s + m \deg p,
\end{equation}
and
\begin{equation} \label{eq:degrad1}
\deg \mathrm{rad} \, f = 2+\deg \mathrm{rad} \, p \leq 2 + \deg p.
\end{equation}
By Lemma~\ref{lem:ABC} with $A = -f(X)$, $B=f(X)-\alpha$ and $C = \alpha$, we have
$$
\deg \mathrm{rad}(f(X)(f(X)-\alpha)) \geq \deg f + 1,
$$
and so by \eqref{eq:deg1} and \eqref{eq:degrad1}
\begin{align} \label{eq:appABC}
\deg \mathrm{rad}(f(X)-\alpha) & \geq r+s+m \deg p + 1 - (2+\deg p) \notag \\
& = r+s+(m-1)\deg p -1.
\end{align}
Now,
\begin{equation} \label{eq:f2}
f^2(X) = (f(X)-\alpha)^r (f(X)-\beta)^s p(f(X))^m,
\end{equation}
so if $f(X)-\alpha$ and $f(X)-\beta$ have between them at least 3 distinct simple roots, then clearly $(f^2,m)$ satisfies the LeVeque condition. If $f(X)-\alpha$ has $u \leq 2$ simple roots, then
$$
\deg \mathrm{rad}(f(X)-\alpha) \leq \frac{\deg f+u}{2},
$$
and so by \eqref{eq:deg1},
$$
\deg \mathrm{rad}(f(X)-\alpha) \leq \frac{u+r+s+m \deg p}{2}.
$$
Since $r,s \geq m/2$, this gives a contradiction with \eqref{eq:appABC} if $m > 2$ or $u = 0$. That is, if $m > 2$ we are done. Otherwise, when $m=2$, $f(X)-\alpha$ and similarly $f(X)-\beta$ each have at least one simple root (distinct from each other and from the roots of $f$). If $\alpha = 0$, then
$$
f^2(X) = X^{r^2}(X-\beta)^{rs} (f(X)-\beta)^s p(X)^4 p(f(X))^2,
$$
and so $(f^2,2)$ satisfies the LeVeque condition, noting that $r$ and $s$ are both odd. Similarly this holds if $\beta=0$, so assume $\alpha,\beta \neq 0$. If $f(X)-\alpha = (X-\gamma)q_1(X)^2$ and $f(X)-\beta=(X-\delta)q_2(X)^2$, then $p$, $q_1$ and $q_2$ are pairwise coprime, so $p q_1 q_2 \mid f'$. This is again a contradiction, since $\deg(p q_1 q_2) \geq 1 + 2(\deg f - 1)/2 > \deg f'$. Thus $f(X)-\alpha$ or $f(X)-\beta$ must have another root of odd multiplicity, and so referring to \eqref{eq:f2}, we conclude once more that $(f^2,2)$ must satisfy the LeVeque condition.
\end{proof}

\section{Diophantine equations with separated variables} \label{sec:BT}

In describing all $f,g \in \Z[X]$ for which $f(x)=g(y)$ has infinitely many integer solutions, Bilu and Tichy \cite{BT} list the following exceptional classes, which they call \emph{standard pairs}. Here $a,b$ are non-zero elements of some field, $m$ and $n$ are positive integers, and $p$ is a non-zero polynomial (which may be constant):
\begin{itemize}
\item $(X^m,aX^rp(X)^m)$ or switched i.e. $(aX^rp(X)^m, X^m)$, where we have $0 \leq r \leq m$, $\gcd(r,m)=1$ and $r+\deg p > 0$.
\item $(X^2, (aX^2+b)p(X)^2)$ (or switched).
\item $(D_m(X,a^n),D_n(X,a^m))$, where $\gcd(m,n)=1$. Here $D_m(X,a)$ denotes the $m$th Dickson polynomial, defined by
$$
D_m(z+a/z,a)=z^m+(a/z)^m.
$$
\item $(a^{-m/2} D_m(X,a),-b^{n/2}D_n(X,b))$, where $\gcd(m,n)=2$.
\item $((aX^2-1)^3,3X^4-4X^3)$ (or switched).
\end{itemize}

Moreover, to deal with the general case of $S$-integers over a number field $K$, they introduce a \emph{specific pair} over $K$, namely
$$
(D_m(X,a^{n/d}),-D_n(X \cos(\pi/d),a^{m/d}))
$$
(or switched), where $d=\gcd(m,n) \geq 3$ and $a,\cos(2\pi/d) \in K$. For $F \in K[X,Y]$ and a subring $R$ of $K$, say that $F(x,y)=0$ has \emph{infinitely many solutions with bounded $R$-denominator} if there exists a non-zero $D \in R$ such that $F(x,y)=0$ has infinitely many solutions $(x,y) \in K \times K$ with $D x, D y \in R$. Then we have the following \cite[Theorem~10.5]{BT}.

\begin{theorem}[Bilu,Tichy] \label{thm:BiluTichy}
Let $K$ be a number field, $S$ a finite set of places of $K$ containing all archimedean places, and $f,g \in K[x]$. Then the following are equivalent:
\begin{enumerate}
\item[(a)] The equation $f(x)=g(y)$ has infinitely many solutions with a bounded $\mathfrak{o}_S$-denominator.
\item[(b)] We have $f=\varphi \circ f_1 \circ \lambda$ and $g = \varphi \circ g_1 \circ \mu$, where $\lambda, \mu \in K[X]$ are linear polynomials, $\varphi \in K[X]$, and $(f_1,g_1)$ is a standard or specific pair over $K$ such that the equation $f_1(x)=g_1(y)$ has infinitely many solutions with bounded $\mathfrak{o}_S$-denominator.
\end{enumerate}
If $K$ is totally real and $S$ is the set of archimedean places (in particular, if $K=\Q$ and $\mathfrak{o}_S=\Z$), then the specific pair may be omitted in (b).
\end{theorem}

Given polynomials $f,g \in \mathfrak{o}_S[X]$ which are neither linear nor monomials, define
$$
\cE(f) := \{ 1 \} \cup \{ \ell \in \Z_{\geq 2} : (f,\ell) \text{ fails the LeVeque condition} \},
$$
and note that the set
$$
\cE(f,g) := \{ (k,\ell) \in \Z^+ \times \Z^+ : k \in \cE(g), \ell \in \cE(f) \text{ and } \gcd(k,\ell)=1 \}
$$
is finite. We have the following result, of which Theorem~\ref{thm:MDBT1} is a special case.

\begin{theorem}
Let $K$ be a number field, $S$ a finite set of places containing all the archimedean ones, and suppose $f,g \in \mathfrak{o}_S[X]$ have at least two distinct roots. Then there are only finitely many $(x,y) \in \mathfrak{o}_S^2$ such that either $f(x)=g(y)$ or $(f(x),g(y))$ is multiplicatively dependent of rank 1 unless for some $(k,\ell) \in \cE(f,g)$, we have $f(X)^k = \varphi(f_1(\lambda(X)))$ and $\zeta g^\ell = \varphi(g_1(\mu(X)))$, where $\zeta$ is a root of unity contained in $K$, $\lambda, \mu \in K[X]$ are linear, $\varphi \in K[X]$, and $(f_1,g_1)$ is a standard or specific pair over $K$, such that $f_1(x)=g_1(y)$ has infinitely many solutions with a bounded $\mathfrak{o}_S$ denominator. If $K$ is totally real and $S$ is the set of archimedean places (in particular, if $K=\Q$ and $\mathfrak{o}_S=\Z$), then the specific pair may be omitted.
\end{theorem}

\begin{proof}
Suppose $(x,y) \in \mathfrak{o}_S^2$ is such that $f(x)=g(y)$ or $(f(x),g(y))$ is multiplicatively dependent of rank 1. Then we must have
\begin{equation} \label{eq:BTmultdep}
f(x)^{k'} = g(y)^{\ell'}
\end{equation}
for some $k',\ell' \neq 0$. Recall that by assumption, both $f$ and $g$ have at least 2 distinct roots. Hence, for example by \cite[Proposition~1.5]{KLS}, there are only finitely many pairs $(x,y) \in \mathfrak{o}_S^2$ such that $f(x)$ and $g(y)$ are both $S$-units. Excluding these, we may assume, without loss of generality, that there is some $v \notin S$ such that $|f(x)|_v < 1$. Thus, if $k',\ell'$ have different signs in \eqref{eq:BTmultdep}, then
$$
1 \geq |g(y)|_v = |f(x)|_v^{k'/\ell'} > 1,
$$
a contradiction. Therefore, we may assume $k',\ell' > 0$. We have
$$
f(x)^k = \zeta g(y)^\ell
$$
where $k = k'/\gcd(k',\ell')$ and $\ell=\ell'/\gcd(k',\ell')$ are coprime, and $\zeta$ is a root of unity contained in $K$. If $(k,\ell) \notin \cE(f,g)$, assume without loss of generality that $k > 1$ and $(g,k)$ satisfies the LeVeque condition. Take $a,b \in \Z$ such that $ak+b\ell=1$. Then
\begin{equation} \label{eq:PerfPower}
g(y) = g(y)^{ak+b\ell} = (\zeta^{-b} g(y)^a f(x)^b)^k.
\end{equation}
Clearly $z:=\zeta^{-b} g(y)^a f(x)^b \in K$, and in fact $z \in \mathfrak{o}_S$, since for $v \notin S$, $|z|_v = |g(y)|_v^{1/k} \leq 1$. Thus by Theorem~\ref{thm:Superell}, we have $h(y), k \ll 1$, with the implied constants depending only on $K,S$ and $g$. So by Northcott's theorem, there are only finitely many possible choices for $y$. If additionally $\ell > 1$ and $(f,\ell)$ satisfies the LeVeque condition, we similarly obtain that there are only finitely many possibilities for $x$. Otherwise, we have $\ell \in \cE(f)$ and so $\ell \ll 1$ (depending only on $f$). Therefore $f(x)^k=\zeta g(y)^\ell$, with $k,\ell \ll 1$ and $\zeta,y$ belonging to finite sets, and so again there are only finitely many possibilities for $x$.

We conclude that there are infinitely many $(x,y) \in \mathfrak{o}_S^2$ such that $(f(x),g(y))$ is multiplicatively dependent of rank 1 if and only if there are infinitely many solutions $x,y \in \mathfrak{o}_S$ to $f(x)^k=\zeta g(y)^\ell$ for some $(k,\ell) \in \cE(f,g)$ and some root of unity $\zeta \in K$. The result follows by Theorem~\ref{thm:BiluTichy}. 
\end{proof}

\section{Multiplicative dependence in orbits} \label{sec:GTZ}

\subsection{Specialisation} To prove Theorem~\ref{thm:MDGTZ}, we can restrict our attention to the finitely generated (over $\Q$) subfield $K$ of $\C$ generated by the roots and coefficients of $f$ and $g$ together with the points $x$ and $y$. We are thus able to reduce to the case where $K$ is a number field by inductively making appropriate specialisations of the transcendental generators of $K$. This allows us to use the tools from \textsection \ref{subsec:BBGMOS}. For a polynomial $f$ defined over a field $K$, and a homomorphism $\varphi$ whose domain is a subring of $K$ containing the coefficients of $f$, let $f^\varphi$ denote the polynomial obtained by applying $\varphi$ to the coefficients of $f$.

\begin{prop} \label{prop:Specialisation}
Let $K$ be a finitely generated field over $\Q$, suppose $f,g \in K[X]$ are not linear and split into linear factors over $K$, and let $x,y \in K$. Let $E$ be a subfield of $K$ such that $\mathrm{trdeg}(K/E)=1$ and $E/\Q$ is finitely generated. Then there exists a subring $R$ of $K$, a finite extension $E'$ of $E$ and a homomorphism $\varphi : R \to E'$ such that
\begin{enumerate}
\item $R$ contains $x,y$ and every coefficient and root of $f$ and $g$, the leading and lowest order coefficients of $f$ and $g$ have non-zero image under $\varphi$, and distinct roots of $f$ and $g$ have pairwise distinct images under $\varphi$;
\item $\varphi(x)$ is not preperiodic for $f^\varphi$;
\item If $k,\ell \in \Z^+$ and $(f(X),g(X))$ is not of the form $(X^s \tilde f(X)^\ell, X^t \tilde g(X)^k)$ for some $s,t \geq 0$ and polynomials $\tilde f, \tilde g$, then neither is $(f^\varphi(X),g^\varphi(X))$.
\end{enumerate}
\end{prop}

To prove this, we recall the usual setup for specialisation, following the exposition in \cite[\textsection 6]{GTZ1}. We may assume (replacing it with a finite extension if necessary) that $E$ is algebraically closed in $K$. Let $C$ be a smooth projective curve over $E$ whose function field is $K$, and let $\pi : \P_C^1 \to C$ be the natural fibration. Any $z \in \P_K^1$ gives rise to a section $Z : C \to \P^1$ of $\pi$, and for $\varphi \in C(\overline E)$, we let $z_\varphi := Z(\varphi)$, and let $E(\alpha)$ be the residue field of $K$ at the valuation corresponding to $\varphi$. Let $R$ be the valuation ring for this valuation, $E'=E(\alpha)$, and define the homomorphism $R \to E'$ by $z \mapsto z_\varphi$. The polynomial $f \in K[X]$ extends to a rational map (of $E$-varieties) from $\P_C^1$ to itself, whose generic fibre is $f$, and whose fibre above any $\varphi \in C$ is $f^\varphi$. Note that $f^\varphi$ is a morphism of degree $\deg(f)$ from the fibre $(\P_C^1)_\varphi = \P_{E(\varphi)}^1$ to itself whenever the coefficients of $f$ have no poles or zeros at $\varphi$; hence it is a morphism on $\P_{E(\varphi)}^1$ of degree $\deg(f)$ at all but finitely many $\varphi$.

Let $h_C$ be the logarithmic Weil height on $C$ associated to a fixed degree-one ample divisor. We have the following.

\begin{prop} \cite[Proposition~6.2]{GTZ1} \label{prop:spec2}
Let $z \in \P^1_K$. There exists $c > 0$ such that, for $\varphi \in C(\overline E)$ with $h_C(\varphi) > c$, the point $z_\varphi$ is not preperiodic for $f^\varphi$.
\end{prop}

\begin{proof}[Proof of Proposition \ref{prop:Specialisation}]
Let $\psi : C \to \P_E^1$ be any non-constant rational function. By \cite[Prop. 4.1.7]{Lang}, there are positive constants $A$ and $B$ such that for all $P \in \P^1(\overline E)$, the preimage $\varphi = \psi^{-1}(P)$ satisfies $h_C(\varphi) \geq Ah(P)+B$. Since there are infinitely many $P \in \P^1(E)$ such that $h(P) > (c-B)/A$, we obtain infinitely many $\varphi \in C(\overline E)$ such that $h_C(\varphi) > c$. Proposition~\ref{prop:spec2} thus implies that there are infinitely many $\varphi$ satisfying (2), and all but finitely many of these satisfy (1) as well. Finally, it is clear that a specialisation $\varphi$ satisfying (1) also satisfies (3), since by construction all the multiplicities of the roots of $f$ and $g$ are preserved.
\end{proof}

\subsection{Proof of Theorem~\ref{thm:MDGTZ}}

We are now ready to prove the following result, which together with Theorem~\ref{thm:GTZ} will imply Theorem~\ref{thm:MDGTZ}.

\begin{prop} \label{prop:MDGTZprelim}
Let $f,g \in \C[X]$ be polynomials which are not linear or monomials, and let $x,y \in \C$. Then
\begin{enumerate}
\item[(a)] There are only finitely many pairs $(z,w) \in \cO_f(x) \times \cO_g(y)$ such that $z^k=w^\ell$ with $k,\ell$ non-zero integers of different signs.
\item[(b)] If $(z,w) \in \cO_f(x) \times \cO_g(y)$ is such that $z^k=\zeta w^\ell$ with $k,\ell \in \Z^+$ coprime and $\zeta$ a root of unity, then $k, \ell \ll 1$, with the implied constants depending only on $f,g,x,y$.
\item[(c)] Let $k,\ell \in \Z^+$ be relatively prime. Suppose that $(f(X),g(X))$ is not of the form $(X^s \tilde f(X)^\ell, X^t \tilde g(X)^k)$ for some $s,t \geq 0$ and $\tilde f, \tilde g \in \C[X]$, and additionally that $(f^2(X),g(X))$ is not of this form if $k=2$, and that $(f(X),g^2(X))$ is not of this form if $\ell=2$. Then there are only finitely many pairs $(z,w) \in \cO_f(x) \times \cO_g(y)$ such that $z^k = w^\ell$.
\end{enumerate}
\end{prop}

\begin{proof}
Write $f(X) = a_sX^s + \cdots + a_0$ and $g(X) = b_tX^t + \cdots + b_0$. Let 
$$
K = \Q \left( x,y,a_0,\ldots,a_s,b_0,\ldots,b_t, f^{-1}(0),g^{-1}(0),f^{-2}(0),g^{-2}(0) \right). 
$$
We proceed by induction on $\mathrm{trdeg}(K/\Q)$. For the base case, where $K$ is a number field, let 
$$
S = M_K^\infty \cup \left \{ v \in M_K^0 : \min_{0 \leq i \leq s, 0 \leq j \leq t} \{ v(a_i), v(b_j), v(x),v(y) \} < 0 \right \}.
$$
Then for all $n,m \in \Z^+$, $f^n(x), g^m(y) \in \mathfrak{o}_S$. By \cite[Proposition~1.6]{KLS}, both $\cO_f(x) \cap \mathfrak{o}_S^*$ and $\cO_g(y) \cap \mathfrak{o}_S^*$ are finite. Thus, if $\cO_f(x) \times \cO_g(y)$ is infinite, and $(z,w) \in \cO_f(x) \times \cO_g(y)$ is such that $z^k=w^\ell$ with $k,\ell$ integers of different signs, then upon omitting finitely many possibilities we may assume that $|z|_v < 1$ for some $v \in S$, and hence
$$
1 \geq |w|_v = |z|_v^{k/\ell} > 1,
$$
a contradiction. This gives (a) in this case.

Suppose $k,\ell \in \Z^+$ are coprime, and that $(z,w)=(f^n(x),g^m(y))$ satisfies $z^k=\zeta w^\ell$ for a root of unity $\zeta$. By the same argument which follows \eqref{eq:PerfPower}, we have $z=\alpha^\ell$ and $w=\beta^k$ for some $\alpha,\beta \in \mathfrak{o}_S$. Then the first conclusion of Theorem~\ref{thm:Superell} gives $k,\ell \ll 1$ as required for (b). Moreover, if $\ell > 1$, and $f(X)$ (also $f^2(X)$ if $\ell=2$) is not of the form $X^s \tilde f(X)^\ell$ for some $s \geq 0$ and $\tilde f \in \overline K[X]$, then by Lemma~\ref{lem:IterLeVeque}, $(f^j,\ell)$ satisfies the LeVeque condition for some $j \leq 6$. Thus, by the second part of Theorem~\ref{thm:Superell}, we have $h(f^{n-j}(x)) \ll 1$. In particular either $x$ is preperiodic for $f$ or $n \ll 1$. Moreover, noting that $h(g^m(y)) = k h(f^n(x))/\ell$, we see also that either $y$ is preperiodic for $g$ or $m \ll 1$. The same conclusion holds (swapping $f$ and $g$) if $k > 1$. This proves (c), completing the base case.

For the inductive step, let $E$ be a subfield of $K$ with $\mathrm{trdeg}(K/E) = 1$, and $E/\Q$ finitely generated. First suppose that there are infinitely many pairs $(z,w) \in \cO_f(x) \times \cO_g(y)$ such that $z^k=w^\ell$ with $k,\ell$ non-zero integers of different signs. Without loss of generality assume that $x$ is not preperiodic for $f$. Let $R$, $E'$ and $\varphi : R \to E'$ be as in Proposition~\ref{prop:Specialisation}. Properties (1) and (2) imply that there are infinitely many pairs $(z,w) \in \cO_{f^\varphi}(\varphi(x)) \times \cO_{g^\varphi}(\varphi(y))$ such that $z^k=w^\ell$ with $k,\ell$ non-zero integers of different signs. Thus, by the induction hypothesis, $f^\varphi$ or $g^\varphi$ is linear or a monomial, contradicting (1). This proves (a).

For (b), suppose there is a sequence $\{ (z_j,w_j) \}_{j =1}^\infty \subset \cO_f(x) \times \cO_g(y)$, roots of unity $\zeta_j$ and coprime $k_j,\ell_j \in \Z^+$ with $\max(k_j,\ell_j) \to \infty$ as $j \to \infty$, such that $z_j^{k_j}=\zeta_j w_j^{\ell_j}$. The same argument as above implies the existence of such a sequence in $\cO_{f^\varphi}(\varphi(x)) \times \cO_{g^\varphi}(\varphi(y))$, again contradicting the induction hypothesis. An analogous argument also proves (c) in light of property (3) (changing $R$, $E'$ and $\varphi$ to respect the pairs $(f^2,g)$ or $(f,g^2)$ when necessary), completing the proof.
\end{proof}

\begin{proof}[Proof of Theorem~\ref{thm:MDGTZ}]
Suppose, as in the statement of the theorem, that we have $f^i(X)=X^s \tilde f(X)^\ell$, $g^j(X)=X^t \tilde g(X)^k$, and that $\hat f$ and $\hat g_\xi$ share a common iterate, say $\hat f^n = \hat g_\xi^m$, where $\hat f(X) = X^s \tilde f(X^\ell)$, $\hat g(X)=X^t \tilde g(X^k)$. Then for any $N \in \Z^+$
\begin{equation} \label{eq:Semiconjugacy}
f^{iN}(X^\ell) = \hat f^N(X)^\ell, \quad g^{jN}(X^k) = \hat g^N(X)^k.
\end{equation}
Let $x \in \C$ be non-preperiodic for $f$, and let $y$ be an $\ell$-th root of $x$. Then for any $r \in \Z^+$, using \eqref{eq:Semiconjugacy} we obtain
\begin{align*}
f^{inr}(x)^k & = f^{inr}(y^\ell)^k = \hat f^{nr}(y)^{k\ell} \\
& = \hat g_\xi^{mr}(y)^{k\ell} = \xi^{k\ell} \hat g^{mr}(\xi^{-1} y)^{k\ell} \\
& = \xi^{k\ell} g^{jmr}(\xi^{-k} y^k)^\ell,
\end{align*}
and so we have infinitely many pairs $(f^{inr}(x), g^{jmr}(\xi^{-k} y^k))$, $r \in \Z^+$, which are multiplicatively dependent of rank 1.

For the converse, write $f(X) = a_sX^s + \cdots + a_0$, $g(X) = b_tX^t + \cdots + b_0$, and let $K = \Q(x,y,a_0,\ldots,a_s,b_0,\ldots,b_t)$. Suppose $(z,w)=(f^n(x),g^m(y)) \in \cO_f(x) \times \cO_g(y)$ is multiplicatively dependent of rank 1. Then $z^{k'}=w^{\ell'}$ for some non-zero integers $k',\ell'$. By Proposition~\ref{prop:MDGTZprelim}~(a), we may assume without loss of generality that $k',\ell' > 0$. Then $z^k = \zeta w^\ell$, where $k=k'/\gcd(k',\ell,)$, $\ell = \ell'/\gcd(k',\ell')$ are coprime, and $\zeta$ is one of finitely many roots of unity contained in $K$. By Proposition~\ref{prop:MDGTZprelim}~(b), we have $k,\ell \ll 1$. Hence, if there are infinitely many pairs $(z,w) \in \cO_f(x) \times \cO_g(y)$ which are multiplicatively dependent of rank 1, then infinitely many of them satisfy $z^k = \zeta w^\ell$ for some fixed coprime $k,\ell \in \Z^+$ and root of unity $\zeta \in K$. By Proposition~\ref{prop:MDGTZprelim}~(c), this cannot be the case unless $(f(X),g_\zeta(X))$ is of the form $(X^s \tilde f(X)^\ell, X^t \tilde g(X)^k)$ for some $s,t \geq 0$ and $\tilde f, \tilde g \in \C[X]$. In this scenario, we have the semiconjugacies $f(X^\ell)=\hat f(X)^\ell$ and $g_\zeta(X^k)=\hat g(X)^k$, where $\hat f(X) = X^s \tilde f(X^\ell)$ and $\hat g(X) = X^t \tilde g(X^k)$. Then $f^n(X^\ell) = \hat f^n(X)^\ell$, $g_\zeta^m(X^k) = \hat g(X)^k$, and we can write $z^k=\zeta w^\ell$ as
$$
\hat f^n(\hat x)^{k \ell} = \hat g^m(\hat y)^{k \ell},
$$
where $\hat x^\ell = x$ and $\hat y^k=\zeta y$. We conclude that
$$
\hat f^n(\hat x)=\hat g_{\xi}^m(\xi \hat y),
$$
where $\xi$ is another root of unity. The result then follows from Theorem~\ref{thm:GTZ}.
\end{proof}

\section{Sparseness of multiplicatively dependent iterates} \label{sec:Sparse}

\subsection{Rigid divisibility sequences} \label{subsec:RDS} Recall that a sequence $(a_n)_{n=1}^\infty$ of integers is a \emph{divisibility sequence} if whenever $m \mid n$, we have $a_m \mid a_n$. We say that a divisibility sequence $(a_n)_{n=1}^\infty$ is a \emph{rigid divisibility sequence} if additionally, for every prime $p$, there is an exponent $s_p$ such that for each $n \in \Z^+$, either $p \nmid a_n$ or $p^{s_p} \| a_n$. Rigid divisibility sequences generated by polynomial iteration have been studied in the context of the density of prime divisors \cite{J} as well as the existence of primitive prime divisors (see \textsection \textsection \ref{subsec:PPD}) in arithmetic dynamics. It turns out such sequences also provide examples where it is not too difficult to show that the sets $M_{F,\bm{x}}(N)$, defined before the statement of Theorem~1.7, are sparse. We have the following examples \cite[Propositions~3.2~and~3.5]{R} (see also \cite[Theorems~1~and~3]{R2}).

\begin{prop}
Suppose $f \in \Z[X]$ is a monic polynomial with linear coefficient 0. Then $\{ f^n(0) \}_{n \geq 1}$ is a rigid divisibility sequence.
\end{prop}

\begin{prop}
Suppose $f \in \Z[X]$ is monic, and that $\{ f^n(0) \}_{n \geq 1}$ is a rigid divisibility sequence. If $r \in \Z$ is a root of $f$, and $g(X):=f(X+r)-r$, then $\{ g^n(0) \}_{n \geq 1}$ is a rigid divisibility sequence.
\end{prop}

We note that the largest square-free factor of a polynomial value generally grows with the input.

\begin{lemma} \label{lem:LrgSqFree}
Let $f \in \Z[X]$ be such that $(f,2)$ satisfies the LeVeque condition. Then for $x \in \Z$, the largest square-free factor $Q^+(f(x))$ of $f(x)$ satisfies
$$
\log Q^+(f(x)) \gg \log \log |x|,
$$
where the implied constant depends only on $f$.
\end{lemma}

\begin{proof}
Let $x \in \Z$ and write
$$
f(x) = \pm p_1^{2a_1} \cdots p_k^{2a_k} q_1^{2b_1+1} \cdots q_\ell^{2b_\ell+1}
$$
for distinct primes $p_1,\ldots,p_k,q_1,\ldots, q_\ell$ and integers $a_1,\ldots,a_k \geq 1$, $b_1,\ldots,b_\ell \geq 0$. Then
$$
f(x) = \pm p_1^2 \cdots p_k^2 q_1 \cdots q_\ell y^2, \quad y \in \Z.
$$
Note that $Q^+(f(x)) = p_1 \cdots p_k q_1 \cdots q_\ell$,
so applying Theorem~\ref{thm:Superell} with $S=\{ \infty \}$ and $b=\pm p_1^2 \cdots p_k^2 q_1 \cdots q_\ell$, we have $N_S(b)=|b| \leq Q^+(f(x))^2$, and so
$$
\log |x| \ll Q^+(f(x))^{32r^3}.
$$
The result follows upon taking logarithms.
\end{proof}

We are now able to prove Theorem~\ref{thm:SparseIter}.

\begin{proof}[Proof of Theorem~\ref{thm:SparseIter}]
Suppose $(m_1,\ldots,m_n) \in M_{F, \bm{x}}(N)$ with $m_i \geq N^{1/2}+2$ for each $i$. Then
$$
f_1^{m_1}(x_1)^{k_1} \cdots f_n^{m_n}(x_n)^{k_n} = 1
$$
for some $k_1,\ldots,k_n \in \Z$, not all 0. Discard the $O(N^{n-1})$ choices of $(m_1,...,m_n)$ for which the multiplicative dependence relation above is of rank 0, and suppose $i$ is such that $k_i \neq 0$. Then each prime $p \mid f_i^{m_i}(x_i)$ divides some other $f_j^{m_j}(x_j)$. Let $Q$ denote the largest square-free factor of $f_i^{m_i}(x_i)$. Then $Q \geq Q^+(g(f_i^{m_i-2}(x_i))$, where $g=\mathrm{rad}(f_i^2)$ is the product of the distinct irreducible factors (over $\Z$) of $f_i^2$. Since $f_i$ is not linear or a monomial, $f_i^2$ has at least 3 distinct roots (for example by \cite[Lemma~3.2]{Y}). Thus $(g,2)$ satisfies the LeVeque condition, and so by Lemma~\ref{lem:LrgSqFree},
\begin{equation} \label{eq:Qbound}
\log Q \gg \log \log |f_i^{m_i-2}(x_i)| \gg N^{1/2}.
\end{equation}
For each $j \neq i$, let $Q_j$ be the largest factor of $Q$ dividing $f_j^{m_j}(x_j)$. Then there is some $j$ with $Q_j \geq Q^{1/n}$. For such a $j$, and each prime $p$ dividing $Q_j$, let $s_{j,p}$ be the least integer such that $p \mid f_j^{s_{j,p}}(x_j)$. Then, by the properties of a rigid divisibility sequence, $m_j$ must be a multiple of $s_j := \mathrm{lcm}_{p \mid Q_j} s_{j,p}$. But $Q_j \mid f_j^{s_j}(x_j)$, so
$$
Q^{1/n} \ll |f_j^{s_j}(x_j)| \ll e^{d_j^{s_j}}.
$$
Hence \eqref{eq:Qbound} gives
$$
s_j \gg \log \log Q \gg \log N,
$$ 
and so there are at most $O(N/\log N)$ possibilities for $m_j$. The result follows from letting the other $m_\ell$ run freely over $[N^{1/2}+2,N]$, and noting that there are only $O(N^{n-1/2})$ vectors $(m_1,\ldots,m_n) \in [1,N]^n$ with some $m_\ell < N^{1/2}+2$.
\end{proof}

\subsection{Primitive prime divisors in dynamical sequences} \label{subsec:PPD}

Let $K$ be a number field or function field, let $f \in K(X)$ be a rational function of degree $d > 1$, and let $x \in K$. Let $\p$ be a finite prime of $K$ and denote by $v_\p$ the valuation on $K$ associated to $\p$. For $m \geq 1$, we say that a prime $\p$ of $K$ is a \emph{primitive prime divisor} of $f^m(x)$ if $\p$ divides $f^m(x)$ i.e. $v_\p(f^m(x)) > 0$, but $v_\p(f^k(x)) \leq 0$ for all $k < m$. It is often the case that for all but finitely many $m$, $f^m(x)$ has a primitive prime divisor. Indeed, as noted above this was one of the motivations for studying polynomials which generate rigid divisibility sequences \cite{R}. More recently, many authors have considered the problem in other cases \cite{GN,GNT,K,IS}. We will make use of an amalgamation of these results. We say that a number field $K$ \emph{satisfies the $abc$-conjecture} if for any $\varepsilon > 0$, there exists a constant $C_{K,\varepsilon}>0$ such that for all $a,b,c \in K^\times$ satisfying $a+b=c$, we have $h(a,b,c) < (1+\varepsilon) \mathrm{rad}(a,b,c) + C_{K,\varepsilon}$ (see \cite[Conjecture~3.1]{GNT} for the precise definitions of the height and radical used here).

\begin{theorem} \label{thm:PPD}
Let $K$ be a number field or characteristic zero function field of transcendence degree 1, let $f \in K(X)$ be a rational function of degree $d > 1$ which is not a monomial, and let $x \in K$ be a point which is not preperiodic for $f$. If $K$ is a number field, assume that either $K$ satisfies the $abc$-conjecture, or that $0$ is periodic for $f$ and $f$ does not vanish to order $d$ at 0. Then for all but finitely many positive integers $m$, there is a prime $\p$ of $K$ such that $\p$ is a primitive prime divisor of $f^m(x)$.
\end{theorem}

\begin{proof}
This follows from \cite[Theorem~1.1]{GNT} and \cite[Theorem~7]{IS}.
\end{proof}

We conclude by proving Theorem~\ref{thm:MDZsig}.

\begin{proof}[Proof~of~Theorem~\ref{thm:MDZsig}]
By Theorem~\ref{thm:PPD}, there exists a constant $B>0$, depending only on $f$ and $x$, such that $f^m(x)$ has a primitive prime divisor for all $m > B$. Suppose $(m_1,\ldots,m_n) \in M_{F,\bm{x}}(N)$ is such that $m_i > B$ and $f^{m_i}(x)$ is not a root of unity for each $i$, and $m_i \neq m_j$ for all $i \neq j$. Then
$$
f^{m_1}(x)^{k_1} \cdots f^{m_n}(x)^{k_n} = 1
$$
for some $k_1,\ldots,k_n \in \Z$, not all 0. By construction, the multiplicative relation above is not of rank 0. Hence, taking the largest $m_i$ such that $k_i \neq 0$, each prime $\p$ dividing $f^{m_i}(x)$ must divide some $f^{m_j}(x)$ with $j \neq i$. This is a contradiction since $f^{m_i}(x)$ has a primitive prime divisor. Since there are $O(N^{n-1})$ choices of $(m_1,\ldots,m_n)$ with either $m_i \leq B$ for some $i$, $f^{m_i}(x)$ a root of unity for some $i$, or $m_i=m_j$ for some $i \neq j$, the result follows.
\end{proof}

\Address

\end{document}